\documentclass[12pt]{article}
\RequirePackage{amsmath,amsthm,amsfonts,mathrsfs,amssymb,hyperref}
\usepackage{xcolor}
\usepackage{graphicx,stmaryrd,enumitem}
\usepackage{dsfont,cite}
\newcommand{\ind}{1\hspace{-0.098cm}\mathrm{l}}

\def\dd{\mbox{d}}

\def\P{\mathbb{P}}
\def\R{\mathbb{R}}

\newtheorem{thm}{Theorem}

\newtheorem{cor}[thm]{Corollary}
\newtheorem{lem}[thm]{Lemma}

\renewcommand{\geq}{\geqslant}
\renewcommand{\leq}{\leqslant}
\usepackage{tikz,authblk}

\def\parxi{\xi}
\def\ggz{g_2} % {g_2(\parxi\sqrt{\gamma}z)}
\def\gge{g_1} % {\gge}

\title{Persistence probabilities of MA(1) sequences with Laplace innovations and $q$-deformed zigzag numbers}
\author[1]{Frank Aurzada}
\author[2]{Kilian Raschel}
\affil[1]{Technical University of Darmstadt, Germany \href{mailto:aurzada@mathematik.tu-darmstadt.de}{\texttt{aurzada@mathematik.tu-darmstadt.de}}}
\affil[2]{CNRS, Université d'Angers, France, \href{mailto:raschel@math.cnrs.fr}{\texttt{raschel@math.cnrs.fr}}}

\date{\today}

\begin{document}
\maketitle
\allowdisplaybreaks

\begin{abstract}
We study the persistence probabilities of a moving average process of order one with innovations that follow a Laplace distribution. The persistence probabilities can be computed fully explicitly in terms of classical combinatorial quantities like certain $q$-Pochhammer symbols or $q$-deformed analogues of Euler's zigzag numbers, respectively. Similarly, the generating functions of the persistence probabilities can be written in terms of $q$-analogues of the exponential function or the $q$-sine/$q$-cosine functions, respectively.
\end{abstract}


\section{Introduction and main results}
The purpose of this paper is to establish a connection between two quite distinct mathematical objects: On the one hand, we have certain well-studied probabilistic quantities, namely persistence (i.e.\ non-exit) probabilities for stochastic processes. At the other end of the connection, we find certain classical combinatorial quantities, namely Euler's zigzag numbers and the $q$-analogues of them introduced by Stanley (cf.\ \cite{St-10}).

The persistence probability related to a one-dimensional stochastic process is simply the probability that the process does not change sign in a finite time frame. These quantities are fundamental building blocks in many applications, be it within probability theory in the analysis of complicated probabilistic models or be it in insurance, finance, etc. Often, these probabilities are studied with a motivation coming from theoretical physics, and for this point of view we refer to the survey papers \cite{Bray2013Persistence,SALCEDOSANZ20221} and the monograph \cite{Metzler2014FirstPassage}.
A survey of the mathematical literature is provided in \cite{Aurzada2015Persistence}.

In the present paper, the point of view is different: Namely, we prove a tight relation between certain persistence probabilities and classical combinatorial quantities. We stress that we are able to give explicit formulas for the persistence probabilties, contrary to most other works, where the focus is on the persistence exponent, i.e.\ the exponential decay rate, only.

Let us consider the persistence probabilities of a moving average (MA) process of order one with Laplace innovations. More concretely, let $(X_i)$ be i.i.d.\ random variables with generalized Laplace distribution, i.e.\ with density
\begin{equation} \label{eqn:densityform}
\parxi e^x \ind_{x<0} + (1-\parxi) e^{-x}\ind_{x>0}, \qquad x\in\R,
\end{equation}
with an asymmetry parameter $\parxi\in[0,1]$. This includes the case of the classical symmetric Laplace distribution when $\parxi=\frac{1}{2}$, the standard exponential distribution for $\parxi=0$, and the negative of a standard exponential distribution for $\parxi=1$.

For a coupling parameter $\theta\in\R$, let us look at the probability
$$
p_n:=p_n^{\theta,\parxi}:=\P( X_1 > \theta X_0, X_2 > \theta X_1, \ldots, X_n > \theta X_{n-1} ),\qquad n=1,2,\ldots
$$
We shall see that we can give explicit expressions for the $p_n^{\theta,\parxi}$ in terms of classical combinatorial quantities. The cases $\theta>0$ and $\theta<0$ turn out to be qualitatively different. Further, the generating function of the $(p_n)$, $P^{\theta,\parxi}(z):=\sum_{n=1}^\infty p_n^{\theta,\parxi} z^n$, can be expressed in terms of certain $q$-series.

\paragraph{Notation.} In order to formulate the results, we shall need some notation from enumerative combinatorics. First of all, we will use the $q$-Pochhammer symbol 
\begin{equation*}
    (a;q)_n:=\prod_{k=0}^{n-1} (1-a q^k).
\end{equation*}
We further introduce the $q$-factorial as
\begin{equation*}
    [n]_q! := \prod_{k=1}^{n} \frac{1-q^k}{1-q}  = \frac{(q;q)_n}{(1-q)^n},
\end{equation*}
with the continuous extension $[n]_1!=n!$. Finally, we define the following $q$-analogues of the classical exponential and trigonometric functions:
\begin{equation} \label{eqn:qversionstrigexp}
e_q(z) := \sum_{n=0}^\infty \frac{z^{n}}{[n]_q!},\quad
    \cos_q(z) := \sum_{n=0}^\infty \frac{(-1)^nz^{2n}}{[2n]_q!},
    \quad\text{and}\quad
    \sin_q(z) := \sum_{n=0}^\infty \frac{(-1)^nz^{2n+1}}{[2n+1]_q!}.
\end{equation}

These functions, the generating function $P^{\theta,\parxi}$, and all other generating functions below are well-defined at least for all $z\in\mathbb{C}$ with $|z|<1$, which will be the general assumption on $z$ from now. 

We are now prepared to formulate the main results. The cases $\theta>0$ and $\theta<0$ have to be handled separately. We remark that the case $\theta=0$ is trivial and we simply have $p_n^{0,\parxi} = (1-\parxi)^n$.

\paragraph{The case $\theta>0$.}
In the case $\theta>0$, the persistence probabilities can be written in terms of $q$-Pochhammer symbols, and their generating functions can be written in terms of the $q$-analogue of the exponential function.

\begin{thm} \label{thm:maintheoremthetapos}
Let  $\theta>0$ and $\parxi\in[0,1]$.
\begin{enumerate}[label={\rm(\alph{*})},ref={\rm(\alph{*})}]
    \item\label{thm1-ita}We have $p_n^{\theta,\parxi}=p_n^{1/\theta,1-\parxi}$ for all $n$.
    \item\label{thm1-itb} For $\theta=1$,  $p_n^{1,\parxi}=\frac{1}{(n+1)!}$.
    \item\label{thm1-itc} For $\parxi=1$, $p_n^{\theta,1}=\frac{\theta^{n(n+1)/2}}{[n+1]_\theta!}$.
    % $p_n^{\theta,1}=\frac{\theta^{n(n+1)/2}}{(\theta;\theta)_{n+1}}\, (1-\theta)^{n+1}$. 
    For  $\parxi=0$, we have $p_n^{\theta,0}=\frac{1}{[n+1]_\theta!}$.
    \item\label{thm1-itd} For $\theta\neq 1$ and $\parxi\in[0,1)$, the persistence probabilities are given by
\begin{equation} \label{eqn:ppforthetappositive}
p_n^{\theta,\parxi}=\frac{(-\frac{\parxi}{1-\parxi};\theta)_{n+1}}{(\theta;\theta)_{n+1}}\, (1-\theta)^{n+1}(1-\parxi)^{n+1},\qquad n\geq 1.
\end{equation}
\item\label{thm1-ite} For any $\theta>0$, the generating function of the $(p_n)$ equals
\begin{equation}
\label{eq:formula_GF_theta>0}
P(z):=P^{\theta,\parxi}(z) := \sum_{n=1}^\infty p_{n}^{\theta,\parxi} z^n = \frac{1}{z}\left(  \frac{e_\theta\bigl((1-\parxi)z\bigr)}{e_\theta\bigl(-\parxi z\bigr)}-1\right) -1  .
\end{equation}
\end{enumerate}
\end{thm}
% For the case (b) of the last theorem it is worth remarking that for $\theta=1$ or $\parxi=1$, the canonical (continuity) extension of the formula is valid, cf.\ the case (c) of the theorem and \eqref{eqn:posparxione}.

Let us first remark that in the case of symmetric distributions (meaning $\parxi=\tfrac12$), the invariance \ref{thm1-ita} was already derived in \cite[eq.~(26)]{MaDh-01}. Moreover, the case $\theta = 1$ in item~\ref{thm1-itb} is universal for all continuous distributions, by a simple symmetry argument; see \cite[eq.~(28)]{MaDh-01} for the original statement. Let us stress that this theorem (and also our second main theorem) is one of the rare instances where the persistence probabilities can be expressed explicitly, and in particular the influence of the parameters $\theta$ and $\parxi$ is fully explicit, see item~\ref{thm1-itd}. Let us also mention that the generating function in \eqref{eq:formula_GF_theta>0} may alternatively be expressed as a ratio of two infinite Pochhammer symbols (see \eqref{eq:expression_P_ratio_Pochhammer}) or in terms of the $q$-hypergeometric series ${}_1\phi_0$. Further, we remark that a small computation using the $q$-binomial theorem shows that $e_{1/\theta}(z)=\frac{1}{e_\theta(-z)}$, which gives the following alternative representation of \eqref{eq:formula_GF_theta>0}:
\begin{equation*}
P^{\theta,\parxi}(z)  = \frac{1}{z}\left(  e_\theta\bigl((1-\parxi)z\bigr)\cdot e_{1/\theta}\bigl(\parxi z\bigr)-1\right) -1;
\end{equation*}
and from this, one immediately obtains the generating function version $P^{\theta,\parxi}=P^{1/\theta,1-\parxi}$ of the probabilistic duality in item \ref{thm1-ita}. Note that the function $e_{1/\theta}(z)$ is the second natural $q$-analogue of the exponential function; it can be expressed as follows:
\begin{equation*}
    E_\theta(z) := e_{1/\theta}(z) =  \sum_{n=0}^\infty \frac{\theta^{n(n-1)/2} z^{n}}{[n]_\theta!}. % = (-(1-\theta)z ; \theta)_\infty. %% infinite q-Pochhammer description only valid for \ţheta<1
\end{equation*}

Note that the formula \eqref{eq:formula_GF_theta>0} in item~\ref{thm1-ite} has four interesting borderline cases: For $\theta=1$, we obtain essentially the classical exponential function and for $\theta=0$ the geometric series:
\begin{equation*}
P^{1,\parxi}(z) =\frac{e^z-1}{z} -1,\qquad\text{and}\qquad P^{0,\parxi}(z)=\frac{1}{1-(1-\parxi)z}-1.
\end{equation*}

%can be seen as a $q$-analogue of the classical exponential function (up to minor modifications), since for $\theta = 1$ the function $e_\theta$ reduces to the usual exponential function so that
% \begin{equation*}
% P^{1,\parxi}(z) =\frac{e^{(1-\parxi)z}-e^{-\parxi z}}{ze^{-\parxi z}} -1 =\frac{e^z-1}{z} -1 .
% \end{equation*}
The other borderline cases are when $\parxi=0$ and $\parxi=1$, respectively, where we have
 \begin{equation*}
 P^{\theta,0}(z) = \frac{1}{z}\bigl(  e_\theta(z)-1\bigr) -1\qquad\text{and}\qquad P^{\theta,1}(z) = \frac{1}{z}\bigl(  e_{1/\theta}(z)-1\bigr) -1,
 \end{equation*}
 so that the generating function in \eqref{eq:formula_GF_theta>0} is a generalization of the $q$-exponential function. 
 
As a final remark, we note that when $\xi\in(0,1)$
 \begin{equation} \label{eqn:oneastart}
 p_n^{\theta,\parxi}\sim c_{\theta,\parxi} \begin{cases} \bigl((1-\theta)(1-\parxi)\bigr)^n, & \theta\in[0,1) \\
 \bigl((1-\frac{1}{\theta}) \parxi \bigr)^n, & \theta>1;
 \end{cases}
 \end{equation}
which yields a simple formula for the persistence exponent.
The above asymptotics for $\theta<1$ follows directly from \eqref{eqn:ppforthetappositive} by noting that the fraction of the $q$-Pochhammer symbols converges to a constant for $n\to\infty$. For $\theta>1$, one additionally uses item \ref{thm1-ita} of Theorem~\ref{thm:maintheoremthetapos}.

\paragraph{The case $\theta<0$.}
In order to formulate the main result in the case $\theta<0$, we need some more notation from combinatorics. Let us consider Euler's zigzag numbers $(E_n)_{n\geq1}$, which can be characterized by
\begin{equation*}
   \sum_{n=1}^\infty \frac{E_n}{n!} z^n :=  \frac{1}{\cos(z)}+\frac{\sin(z)}{\cos(z)}-1 = \frac{2\sin(\frac{z}{2})}{\cos(\frac{z}{2})-\sin(\frac{z}{2})},
\end{equation*}
where the last equality follows from classical trigonometric identities.
The zigzag numbers $(E_n)$ are classical objects in combinatorics, they start with $1$, $1$, $2$, $5$, $16$, $61$, $272$, $\ldots$ and are related e.g.\ to alternating permutations.

In his paper \cite{Stanley1976BinomialPosets} (cf.\ \cite{St-10}), Stanley introduces a $q$-analogue $(E_n(q))_{n\geq 1}$ of these numbers. According to Theorem~2.1 in \cite{St-10}, they can be characterized by
\begin{equation}
\label{eq:Stanley_q_analogue}
   \sum_{n=1}^\infty \frac{E_n(q)}{[n]_q!} z^n := \frac{1}{\cos_q(z)}+\frac{\sin_q(z)}{\cos_q(z)}-1.
\end{equation}
Stanley computes in particular $E_1(q)=E_2(q)=1$, and
\begin{align*}
E_3(q)& = q + q^2,\\
E_4(q) &= q^2 + q^3 + 2q^4 + q^5,\\
E_5(q) &= q^2 + 2q^3 + 3q^4 + 4q^5 + 3q^6 + 2q^7 +q^8 ,\ldots
\end{align*} 
It turns out that the persistence probabilities for $\theta<0$ can be expressed in terms of these $q$-analogues of Euler's zigzag numbers and that the generating function of the persistence probabilities can be expressed in terms of the $q$-analogues of sine and cosine.

\begin{thm}\label{thm:gfthetaneg}
Let $\theta<0$ and $\parxi\in[0,1]$. 
\begin{enumerate}[label={\rm(\alph{*})},ref={\rm(\alph{*})}]\setcounter{enumi}{5}
    \item\label{thm2-itf} We have $p_n^{\theta,\parxi}=p_n^{1/\theta,\parxi}$ for all $n$.
    \item\label{thm2-itg} For $\theta=-1$ and $\parxi=\frac{1}{2}$, we have $P^{-1,\frac{1}{2}}(z)=\frac{\frac{2}{z}\sin(\frac{z}{2})}{\cos(\frac{z}{2})-\sin(\frac{z}{2})}-1$.
    \item\label{thm2-ith} For any $\theta<0$, the generating function of the $(p_n)$ equals
\begin{equation}
\label{eq:main_formula_theta<0}
P^{\theta,\parxi}(z)=\sum_{n=1}^\infty p_n^{\theta,\parxi} z^n = \frac{1}{z \,\sqrt{\parxi(1-\parxi)}}\, \frac{\sin_{-\theta}(\sqrt{\parxi(1-\parxi)} z)}{\cos_{-\theta}(\sqrt{\parxi(1-\parxi)} z)-\sqrt{\frac{1-\parxi}{\parxi}} \sin_{-\theta}(\sqrt{\parxi(1-\parxi)} z)}-1.% \frac{1+g_2(z)}{1+g_1(z)-\frac{z}{2}(1+g_2(z))}-1.
\end{equation}
\item \label{thm2-iti} For $\parxi\in(0,1)$ we have
\begin{equation}
\label{eq:main_formula_theta<0_pn}
% p_n^{\theta,\parxi} = b_{n+1}(\theta,\parxi) \cdot (\parxi(1-\parxi))^{n/2}, \qquad n\geq 1,
p_n^{\theta,\parxi} = (\parxi(1-\parxi))^{n/2}\cdot \sum_{r=1}^{n+1} \left[ \left(\sqrt{\frac{1-\parxi}{\parxi}}\right)^{r-1}
  \ \sum_{\substack{j_1+\dots+j_r=n+1\\ j_k\ \text{odd}}}
  \prod_{k=1}^r \frac{E_{j_k}(-\theta)}{[j_k]_q!}\right], \qquad n\geq 1.
\end{equation}
% where
% $$
% b_n(\theta,\parxi) := \sum_{r=1}^{n} \left[ \left(\sqrt{\frac{1-\parxi}{\parxi}}\right)^{r-1}
%   \ \sum_{\substack{j_1+\dots+j_r=n\\ j_k\ \text{odd}}}
%   \prod_{k=1}^r \frac{E_{j_k}(-\theta)}{[j_k]_q!}\right],
% \qquad n\geq 1.
% $$
% In the symmetric case $\parxi=\tfrac12$, the quantities $p_n$ can be expressed in terms of a $q$-analogue of the Euler zigzag numbers, to be properly defined in \eqref{eq:variation_q_analogue} below:
% \begin{equation*}
% p_n=\frac{E^*_{n+1}(-\theta)}{[n+1]_{-\theta}!},\qquad n\geq 1.
% \end{equation*}
\end{enumerate}
\end{thm}

We remark that, for $\theta<0$, the borderline cases $\parxi=0$ and $\parxi=1$ are trivial because we have $p_n^{\theta,0}=1$ and $p_n^{\theta,1}=0$ for all $n$; these cases are also included as continuous extensions of  \eqref{eq:main_formula_theta<0} and \eqref{eq:main_formula_theta<0_pn}.
% The proof of item~\ref{thm2-itf} is identical to that of item~\ref{thm1-ita}.
The formula in item~\ref{thm2-itg} for $\theta=-1$ and $\parxi=\frac{1}{2}$ is universal for all symmetric distributions, as shown in \cite{MaDh-01} (eq.~(32) and below).
For $\theta=-1$ and general $\parxi\in[0,1]$, it follows from \eqref{eq:main_formula_theta<0} that
\begin{equation} \label{eqn:majumdardharextension}
p_n^{-1,\parxi} =  \parxi (\parxi(1-\parxi))^{n/2} \, \sum_{k\in\mathbb{Z}} \frac{1}{\left(\arctan(\sqrt{\frac{\parxi}{1-\parxi}})+k\pi\right)^{n+2}}.
\end{equation}
Note that this formula implies that $p_n^{-1,\parxi} \sim \parxi  (\parxi(1-\parxi))^{n/2} \arctan\bigl(\sqrt{\frac{\parxi}{1-\parxi}}\bigr)^{-(n+2)}$, as $n\to\infty$.
In the special case $\parxi=1/2$ (where the $\arctan$ becomes $\pi/4$), the formula \eqref{eqn:majumdardharextension} was obtained in \cite[eq.~(32)]{MaDh-01}. We further remark that the expression in \eqref{eqn:majumdardharextension} is in fact a polynomial in $\parxi$; see Lemma~\ref{lem:polynomial_in_xi}.
% We have $E^*_1(q) = 1$, $E^*_2(q) = \frac{1+q }{2}$, and
% \begi^n{align*}
% E^*_3(q) &= \frac{1+  3q + 3q^2 + q^3}{2^2}, \\
% E^*_4(q) &=\frac{1+5q+9q^2+ 10q^3+ 9q^4  +5q^5 +q^6  }{2^3},\\
% E^*_5(q) &=\frac{ 1 + 7q + 19q^2 + 32q^3 + 44q^4 + 50q^5 + 44q^6 + 32q^7 + 19q^8 + 7q^9 + q^{10}}{2^4}.
% \end{align*}
Finally, let us comment that the duality of generating functions corresponding to the duality of the $p_n$ in item \ref{thm2-itf} is equivalent to the well-known formula $\frac{\sin_{1/q}(z)}{\cos_{1/q}(z)}=\frac{\sin_{q}(z)}{\cos_{q}(z)}$ when applied to the generating functions in \eqref{eq:main_formula_theta<0}.

Our final remark in the case $\theta<0$ concerns the asymptotic behaviour of the persistence probabilities. Let $\arctan_q$ denote the $q$-analogue of the classical $\arctan$ function, defined as the inverse of $\tan_q:=\frac{\sin_q}{\cos_q}$ (see \eqref{eq:def_arctan_q} for the precise definition). 
\begin{cor}
\label{cor:asymptotics_theta<0}
    Let $\theta\in[-1,0]$. We have
\begin{equation*}
   \lim_{n\to\infty} \frac{1}{n}\log p_n^{\theta,\parxi} = \frac{\sqrt{\parxi(1-\parxi)}}{\arctan_{-\theta}\bigl(\sqrt{\frac{\parxi}{1-\parxi}}\bigr)}.
\end{equation*}
Since $\sin_q/\cos_q=\sin_{1/q}/\cos_{1/q}$, the analogous result holds for $\theta<-1$.
\end{cor}
In particular, when $\parxi=\frac12$, the persistence exponent equals $\frac{1}{2\arctan_{-\theta}(1)}$, which motivates introducing the following $q$-analogue: $\bigl[\tfrac{\pi}{4}\bigr]_q := \arctan_q(1)$.
%We were, however, unable to locate any references discussing this quantity. Neither could we relate this to the formula in \eqref{eq:main_formula_theta<0_pn}.

\paragraph{Related work.} Persistence probabilities of moving average processes have been the subject of a few studies already. The most important reference for the present paper is \cite{MaDh-01}, where some general facts are collected and the case of symmetric uniformly distributed innovations is studied. Other references are 
\cite{Krishna2016Persistence,aurzadabothe} and \cite{AurzadaRaschel2025a}, that consider the case of Gaussian and general uniform innovations, respectively.

A process that is related to MA sequences is an autoregressive (AR) process. Here, the reference closest to the present paper is \cite{larralde04}, where AR processes with symmetric Laplace innovations are considered. Other references are \cite{novikov08a, novikov08b,baumgarten14,kettner19,BMMi-10, VyWa-23,AlBoRaSi-23}.

In the present paper, we obtain explicit formulas for the persistence probabilities, which is rather unusual. Often, the focus is on persistence exponents, i.e.\ the exponential decay rate, only. An important reference for characterizing persistence exponents of Markov chains via spectral properties is \cite{AuMuZe-21}, which considers both MA and AR processes as well as more general Markov chains.

\paragraph{Extensions.} It seems that an obvious trail for further research is to investigate whether the techniques from this paper can be applied to shifted Laplace distributions or bi-exponential distribution with different tails. The present paper as well as some recent others (e.g.\ \cite{AlBoRaSi-23,AurzadaRaschel2025a}) have shown tight connections between persistence probabilities and classical combinatorial quantities. It seems that this connection can be exploited in both ways also in other setups.

\paragraph{Outline.} This paper is structured as follows: Section~\ref{sec:positive} contains the proofs of the symmetries claimed in the first items of the main theorems, respectively, and the proofs of all the facts in the case $\theta>0$. This is complemented by Section~\ref{sec:negative}, where all the proofs for the results in the case $\theta<0$ are collected.

\section{Proofs for the case \texorpdfstring{$\theta>0$}{of positive theta}} \label{sec:positive}

We start with the proofs of the symmetries for the persistence probabilities, i.e.\ the first items in Theorem~\ref{thm:maintheoremthetapos} and~\ref{thm:gfthetaneg}, respectively.

\begin{proof}[Proof of Theorem~\ref{thm:maintheoremthetapos}\ref{thm1-ita} and Theorem~\ref{thm:gfthetaneg}\ref{thm2-itf}]
    Let $\theta>0$. Then
    \begin{eqnarray*}
        p_n^{\theta,\parxi}
        &=&
        \P( X_i> \theta X_{i-1}, i=1,\ldots, n )
        \\
        &=&
        \P(  \theta^{-1} X_i> X_{i-1}, i=1,\ldots, n )
        \\
        &=&
        \P(  \theta^{-1} (-X_i)< (-X_{i-1}), i=1,\ldots, n )
        \\
        &=&
        \P(  (-X_{i-1}) > \theta^{-1} (-X_i), i=1,\ldots, n)
        \\
        &=&
        \P(  (-X_{n-i}) > \theta^{-1} (-X_{n-i+1}), i=1,\ldots, n)        
        \\
      &=&
        \P( Y_i > \theta^{-1} Y_{i-1}, i=1,\ldots, n )
        \\
        &=&
        p_n^{1/\theta,1-\parxi},
    \end{eqnarray*}
    where we set $Y_i:=-X_{n-i}$, $i=0,\ldots,n$, and observe that the $(Y_0,\ldots,Y_{n})$ are i.i.d.\ Laplace with density $(1-\parxi) e^x \ind_{x<0} + \parxi e^{-x}\ind_{x>0}$, $x\in\R$.
    %If the $X_i$ were Laplace with parameter $q$, the $Y_i$ are Laplace with parameter $1-q$. 
    For $\theta<0$, the computation can be adapted; the sign change from $>$ to $<$ occurs already when dividing by $\theta$ so that no passing from $X_i$ to $-X_i$ is necessary. We can put $Y_i:=X_{n-i}$, $i=0,\ldots,n$, in the above argument and can observe that the $(Y_0,\ldots,Y_{n})$ are i.i.d.\ Laplace with density as in \eqref{eqn:densityform}.
\end{proof}

We can now start with the considerations for the case $\theta>0$. Let us consider the following auxiliary function
\begin{equation}
    \label{eq:auxiliary_function}p_n(z):=\P( X_1 > \theta z, X_2 > \theta X_1, \ldots, X_n > \theta X_{n-1} ),\qquad n=1,2,\ldots.
\end{equation}

Clearly, when $\theta>0$, the function $z\mapsto p_n(z)$ is continuous and decreasing with $p_n(+\infty)=0$ and $p_n(-\infty)=p_{n-1}$ for $n\geq 2$.
We start with treating $p_n(z)$ for positive $z$ and obtain an explicit formula for these functions.

\begin{lem} Let $\theta>0$ and $\parxi\in[0,1]$. Then, for all $n\geq 1$,
\begin{eqnarray*}
p_n(z) & =&  c_n e^{-\alpha_n z},\qquad z\geq 0,
\end{eqnarray*}
where $\alpha_1:=\theta$, $c_1:=1-\parxi$, and
$$
\alpha_n := (\alpha_{n-1}+1)\theta = \sum_{i=1}^{n} \theta^i,\qquad c_n:= \frac{c_{n-1}(1-\parxi)}{\alpha_{n-1}+1} = \left(1-\parxi\right)^n\cdot\prod_{j=1}^{n-1} \frac{1}{\sum_{i=0}^j \theta^i}=\frac{\left(1-\parxi\right)^n}{[n]_\theta!}.
$$
\end{lem}

We note that for $\theta=1$ we have $\alpha_n=n$ and $c_n=\left(1-\parxi\right)^n \frac{1}{n!}$.
%, while for $\theta\neq 1$ we get the formulas
%$$
%\alpha_n = \frac{\theta(1-\theta^n)}{1-\theta},\qquad c_n= \frac{\left(1-\parxi\right)^n(1-\theta)^{n-1}}{\prod_{j=1}^{n-1}(1-\theta^{j+1})} =  \frac{\left(1-\parxi\right)^n(1-\theta)^{n}}{(\theta;\theta)_n}.
%$$

\begin{proof}
The proof is by induction. For $n=1$, the claim follows from $p_1(z)=\P(X_1>\theta z)=(1-\parxi)\,e^{-\theta z}$ so that indeed $c_1=1-\parxi$ and $\alpha_1=\theta$.
% When $n\geq 2$, we start with the observation
% %\begin{eqnarray} %\label{eqn:thetaposzposobservation}
% $$
%     p_n(z) = \int_{\theta z}^\infty p_{n-1}(x_1) e^{-x_1} (1-\parxi) \dd x_1,\qquad z\geq 0.
% $$ %\end{eqnarray}
% Using this and the induction hypothesis, we obtain
% The induction steps follows from \eqref{eqn:thetaposzposobservation}:
When $n\geq 2$, we condition on $X_1$ and use the induction hypothesis to get
\begin{align*}
p_n(z) &=\int_{\theta z}^\infty p_{n-1}(x_1) e^{-x_1} (1-\parxi) \dd x_1 = \int_{\theta z}^\infty c_{n-1} e^{-\alpha_{n-1}x_1} e^{-x_1}(1-\parxi) \dd x_1 
\\
&= \frac{c_{n-1}(1-\parxi)}{\alpha_{n-1}+1} \, e^{-(\alpha_{n-1}+1)\theta z}. \qedhere
\end{align*}
\end{proof}

In the next step, we treat $p_n(z)$ for negative $z$. Since the functions $z\mapsto p_n(z)$ are continuous, in particularly continuously connected at $z=0$, we get from the last lemma that $p_n(0) = c_n$. For negative $z$, we will obtain an explicit formula for $p_n(z)$, but some coefficients will be defined through a recursion.

\begin{lem} \label{lem:pnzpsozneg}
Let $\theta>0$ and $\parxi\in[0,1]$. Then, for all $n\geq 1$,
\begin{eqnarray} \label{eqn:explicitpntheatposzneg}
p_n(z) & =& \sum_{j=1}^n \frac{(-1)^{j-1} \parxi^{j-1}r_{n-j+1}}{\prod_{m=1}^{j-1} (\alpha_m + 1)} \left(1 - e^{\alpha_j z} \right)+c_n,
\qquad z\leq 0,
\end{eqnarray}
where $r_1:=\parxi$ and the further $r_n$ are defined via the recursion
\begin{equation} \label{eqn:recursionqnRR}
r_n := \sum_{j=1}^{n-1} \frac{ (-1)^{j-1}\parxi^j r_{n-j} }{\prod_{m=1}^{j-1} (\alpha_m + 1)}  + c_{n-1} \parxi,\qquad n\geq 2.
\end{equation}
\end{lem}

\begin{proof} The proof is by induction. For $n=1$, observe that
$$
p_1(z)=\P( X_1 > \theta z) =\P(X_1 \in(\theta z,0)) +1-\parxi = \parxi(1-e^{\theta z}) +1-\parxi = r_1(1-e^{\alpha_1 z}) + c_1,
$$
as required, because $r_1=\parxi$ and 
$c_1=1-\parxi$. The main observation is
% \begin{equation} % \label{eqn:thetaposznegobservation}
    $$
    p_n(z)
    =
    \int_{\theta z}^0 p_{n-1}(x_1) e^{x_1} \,\parxi\dd x_1 + \int_0^\infty p_{n-1}(x_1) e^{-x_1}\, (1-\parxi)\dd x_1 = \int_{\theta z}^0 p_{n-1}(x_1) e^{x_1} \,\parxi\dd x_1 + p_n(0).
$$
%\end{equation} 
%For $n>1$, using the main observation \eqref{eqn:thetaposznegobservation}
Using this and the induction hypothesis, we have, for $n\geq 2$,
\begin{eqnarray*}
p_n(z)&=& 
\int_{\theta z}^0 p_{n-1}(x_1) e^{x_1} \,\parxi\dd x_1+p_n(0)
\\
&=&
\int_{\theta z}^0 \left( \sum_{j=1}^{n-1}  \frac{(-1)^{j-1}\parxi^{j-1} r_{(n-1)-j+1}}{\prod_{m=1}^{j-1} (\alpha_m + 1)} \left(1 - e^{\alpha_j x_1} \right)+c_{n-1}\right) e^{x_1} \,\parxi\dd x_1+c_n
\\
&=&
 \sum_{j=1}^{n-1}  \frac{(-1)^{j-1}\parxi^{j-1} r_{n-j}}{\prod_{m=1}^{j-1} (\alpha_m + 1)} \parxi  \left( 1-e^{\theta z} - \frac{1}{\alpha_j+1} (1- e^{(\alpha_j +1)\theta z})\right) + c_{n-1}\parxi (1-e^{\theta z} )+c_n
 \\
&=&
 \left[ \sum_{j=1}^{n-1} \frac{(-1)^{j-1} \parxi^{j} r_{n-j}}{\prod_{m=1}^{j-1} (\alpha_m + 1)}  +  c_{n-1}\parxi \right](1-e^{\theta z})
 \\
 &&
 +\sum_{j=1}^{n-1} \frac{(-1)^{j+1-1} \parxi^{j+1-1}r_{n-(j+1)+1}}{\prod_{m=1}^{j} (\alpha_m + 1)} (1- e^{\alpha_{j+1} z})
 +c_n
  \\
&=&
 \left[ \sum_{j=1}^{n-1}  \frac{(-1)^{j-1} \parxi^j r_{n-j}}{\prod_{m=1}^{j-1} (\alpha_m + 1)} +  c_{n-1}\parxi \right](1-e^{\alpha_1 z})
 \\
 &&
 +\sum_{j=2}^{n}  \frac{(-1)^{j-1}\parxi^{j-1} r_{n-j+1}}{\prod_{m=1}^{j-1} (\alpha_m + 1)} (1- e^{\alpha_{j} z})
 +c_n.
\end{eqnarray*}
In the second sum, the terms $j=2,\ldots,n$ already match the required terms in the claim. Finally, the prefactor of $(1-e^{\alpha_1 z})$ is given by $r_n$, as required in \eqref{eqn:explicitpntheatposzneg}, by the definition of the recursion for the $(r_n)$ in \eqref{eqn:recursionqnRR}.
\end{proof}

Recursion \eqref{eqn:recursionqnRR} can be solved for the corresponding generating function, as the next lemma shows.

\begin{lem} \label{eqn:genfunznegthepos}
Let $\theta>0$ and $\parxi\in[0,1]$. Let $R(z):=\sum_{n=1}^\infty r_n z^n$. Then
%$$
%R(z) = \parxi z\cdot \frac{1 +f((1-\parxi)z)}{1+f(-\parxi z)}.
%$$
$$
R(z) = \parxi z\cdot \frac{e_\theta\bigl((1-\parxi)z\bigr)}{e_\theta\bigl(-\parxi z\bigr)}.
$$
\end{lem}

\begin{proof} 
Define $P_0:=1$ and, for $n\geq 1$, $P_n:=\prod_{m=1}^n (\alpha_m+1)=[n+1]_\theta!$. We multiply \eqref{eqn:recursionqnRR} by $z^n$, sum in $n$, use $c_n=\left(1-\parxi\right)^n / P_{n-1}$, and obtain
\begin{eqnarray*}
R(z)
& =&
\sum_{n=1}^\infty r_n z^n
=
r_1 z +
\sum_{n=2}^\infty \sum_{j=1}^{n-1} (-1)^{j-1} \parxi^{j} r_{n-j} \frac{1}{P_{j-1}}\, z^n + \sum_{n=2}^\infty c_{n-1} \parxi z^n
\\
&=& \parxi z +
\sum_{j=1}^\infty (-1)^{j-1} \parxi^j \frac{1}{P_{j-1}}\, z^j \cdot \sum_{n=j+1}^\infty r_{n-j}  z^{n-j} + \parxi \,\sum_{n=2}^\infty (1-\parxi)^{n-1} \frac{1}{P_{n-2}} z^{n}
\\
&=& \parxi z +
\sum_{j=1}^\infty (-1)^{j-1} \parxi^j \frac{1}{P_{j-1}}\, z^j \cdot \sum_{n=1}^\infty r_{n}  z^{n} + \parxi  z\,\sum_{n=1}^\infty (1-\parxi)^n \frac{1}{P_{n-1}} z^{n}
\\
&=& \parxi z - f(-\parxi z) \cdot R(z) + \parxi z\, f((1-\parxi)z),
\end{eqnarray*}
with $f(z):=\sum_{n=1}^\infty \frac{1}{P_{n-1}} z^n=e_\theta(z)-1$. Rearranging gives the claim.
\end{proof}

We can now use that $p_{n-1}=p_n(-\infty)$ to obtain the next result, which proves part \ref{thm1-ite} of Theorem~\ref{thm:maintheoremthetapos}.

\begin{cor} \label{cor:lemma77}
 Let $\theta>0$ and $\parxi\in[0,1]$.   We have
\begin{equation}
\label{eqn:gapstpa}
p_{n-1} = \sum_{j=1}^n  \frac{(-1)^{j-1}\parxi^{j-1} r_{n-j+1}}{\prod_{m=1}^{j-1} (\alpha_m + 1)} +c_n,\qquad n\geq 2,
\end{equation}
with the $(r_n)$ defined in \eqref{eqn:recursionqnRR}. Further, the generating function of the $(p_n)$ is given by \eqref{eq:formula_GF_theta>0}.
\end{cor}

\begin{proof} Setting $z=-\infty$ in \eqref{eqn:explicitpntheatposzneg} and using $p_{n-1}=p_n(-\infty)$ we obtain \eqref{eqn:gapstpa}.
In order to show the result \eqref{eq:formula_GF_theta>0} for the generating function, let us first treat the case $\parxi=0$.
Here, from Lemma~\ref{lem:pnzpsozneg}, $p_{n-1}=c_n=1/P_{n-1}$, where we recall from the proof of Lemma~\ref{eqn:genfunznegthepos} the notation $P_n=\prod_{m=1}^n (\alpha_m+1)$; so that the claim is proved immediately. When $\parxi\neq 0$, we can argue as follows. Multiplying \eqref{eqn:gapstpa} by $z^n$ and summing in $n$ gives
\begin{align*}
z P(z)
&= \sum_{n=2}^\infty p_{n-1} z^n = \sum_{n=2}^\infty \sum_{j=1}^n  \frac{(-1)^{j-1} \parxi^{j-1} r_{n-j+1}}{P_{j-1}} z^n +\sum_{n=2}^\infty c_n z^n
\\
&= \sum_{n=1}^\infty \sum_{j=1}^n  \frac{(-1)^{j-1} \parxi^{j-1} r_{n-j+1}}{ P_{j-1}} z^n -r_1 z+\sum_{n=1}^\infty c_n z^n - c_1 z
\\
&=  \sum_{j=1}^\infty \frac{(-1)^{j-1} \parxi^{j-1}}{P_{j-1}} z^j \cdot \sum_{n=j}^\infty r_{n-j+1} z^{n-j}-\parxi z +\sum_{n=1}^\infty \frac{(1-\parxi)^n}{P_{n-1}} z^n-(1-\parxi) z
\\
&=  -\frac{1}{\parxi} f(-\parxi z) \cdot R(z) z^{-1} +f((1-\parxi)z)-z
\\
&=  -\frac{1}{\parxi} f(-\parxi z) \cdot \parxi z\cdot \frac{1  + f((1-\parxi)z)}{1+f(-\parxi z)}\, z^{-1}+f((1-\parxi)z)\cdot \frac{1+f(-\parxi z)}{1+f(-\parxi z)}-z
\\
&=  - f(-\parxi z) \cdot \frac{1  +f((1-\parxi)z)}{1+f(-\parxi z)}+f((1-\parxi)z)\cdot \frac{1+f(-\parxi z)}{1+f(-\parxi z)}-z
\\
&=   \frac{f((1-\parxi)z)- f(-\parxi z)}{1+f(-\parxi z)}-z, 
\end{align*}
where, as in the proof of Lemma~\ref{eqn:genfunznegthepos}, we have set $f(z):=e_\theta(z)-1$.
\end{proof}

% We can use the last result for $\theta=1$ and recover a known result. This proves the first half of part \ref{thm1-itd} of Theorem~\ref{thm:maintheoremthetapos}.

% \begin{cor} \label{cor:thetaone}
%     For $\theta=1$, we have $p_{n}=\frac{1}{(n+1)!}$.
% \end{cor}

% \begin{proof}
%     Note that in this case $P_{n-1}=\frac{1}{n!}$ so that $f(z)=\sum_{n=1}^\infty \frac{z^n}{n!} = e^z-1$. This shows
%     \begin{equation*}
%     P(z)=\sum_{n=1}^\infty p_{n} z^n  = \frac{f((1-\parxi)z)- f(-\parxi z)}{z(1+f(-\parxi z))}   -1
% =  \frac{e^{(1-\parxi)z}-e^{-\parxi z}}{z e^{-\parxi z}}   -1
%         = \frac{e^{z}-1}{z}   -1. \qedhere
%     \end{equation*}
% \end{proof}

We can now finally prove part \ref{thm1-itd} of Theorem~\ref{thm:maintheoremthetapos}.

\begin{cor} Let $\theta>0$ and $\theta\neq 1$. Let $\xi\in[0,1)$. Then the persistence probabilities are given by (\ref{eqn:ppforthetappositive}).
% $$
% p_n=\frac{(-\frac{\parxi}{1-\parxi};\theta)_{n+1}}{(\theta;\theta)_{n+1}}\,  (1-\theta)^{n+1}(1-\parxi)^{n+1},\qquad n\geq 1.
% $$
For $0<\theta<1$, the generating function can be expressed by
\begin{equation}
    \label{eq:expression_P_ratio_Pochhammer}
    P(z)
=
\frac{1}{z} \left( \frac{((1-\theta)(-\parxi z);\theta)_\infty}{((1-\theta)(1-\parxi) z;\theta)_\infty} -1\right) - 1.
\end{equation}
\label{cor:nonameimp}
\end{cor}

Let us give a few remarks about the borderline cases. The case $\theta=1$ in item \ref{thm1-itb} of Theorem~\ref{thm:maintheoremthetapos} follows as continuity extension of the formulas in Corollary~\ref{cor:nonameimp}.

% Recall that the case $\theta=1$ is treated separately in Corollary~\ref{cor:thetaone}. Alternatively, it can be seen as continuity extension of the formulas in Corollary~\ref{cor:nonameimp}.

Similarly, we can get item  \ref{thm1-itc} of Theorem~\ref{thm:maintheoremthetapos}: The case $\parxi=1$ follows as a continuity extension from the formulas in Corollary~\ref{cor:nonameimp} for $\parxi<1$ with $\parxi\nearrow 1$ or, alternatively, via the formula $p_n^{\theta,\parxi}=p_n^{1/\theta,1-\parxi}$. The result is
\begin{equation*} % \label{eqn:posparxione}
    p_n^{\theta,1} = p_n^{1/\theta,0} = \frac{\theta^{n(n+1)/2}}{(\theta;\theta)_{n+1}}\, (1-\theta)^{n+1},\qquad n\geq 1.
\end{equation*}
The case $\parxi=0$ is included in Corollary~\ref{cor:nonameimp} and boils down to the formula mentioned in item  \ref{thm1-itc} of Theorem~\ref{thm:maintheoremthetapos}. 

\begin{proof}[Proof of Corollary~\ref{cor:nonameimp}] Let us first consider the case $0<\theta<1$. Recall that $\alpha_m+1=\sum_{i=0}^m \theta^i=\frac{1-\theta^{m+1}}{1-\theta}$. Therefore, $P_n:=\prod_{i=1}^n (1+\alpha_m)=\frac{\prod_{m=1}^n (1-\theta \theta^m)}{(1-\theta)^n} = \frac{(\theta;\theta)_{n+1}}{(1-\theta)^{n+1}}$. Thus, by the $q$-binomial theorem, we have
\begin{eqnarray*}
f(z)&=&\sum_{n=1}^\infty \frac{1}{P_{n-1}}\,z^n = \sum_{n=1}^\infty \frac{(1-\theta)^n}{(\theta;\theta)_n}\,z^n = \sum_{n=0}^\infty \frac{((1-\theta)z)^n}{(\theta;\theta)_n} - 1 = \frac{1}{((1-\theta)z;\theta)_\infty} -1.
% \\
% &=& \frac{1}{\prod_{k=0}^\infty (1-(1-\theta)z \theta^k)}-1.
\end{eqnarray*}
Therefore, using \eqref{eq:formula_GF_theta>0} (i.e.\ Corollar~\ref{cor:lemma77}) in the first step and the $q$-binomial theorem in the fourth step, we obtain
\begin{eqnarray}
P(z) = \sum_{n=1}^\infty p_n z^n \nonumber
%&=& \frac{1}{z} \,\frac{1+f((1-\parxi)z)-(1+f(-\parxi z))}{1+f(-\parxi z)} -1
%\\\nonumber
&=& \frac{1}{z} \left( \frac{1+f((1-\parxi) z))}{1+f(-\parxi z)} -1\right) - 1\nonumber
\\
&=& \frac{1}{z} \left( \frac{((1-\theta)(-\parxi z);\theta)_\infty}{((1-\theta)(1-\parxi) z;\theta)_\infty} -1\right) - 1\label{eq:ration_Poch_infty}
\\
&=&
\frac{1}{z} \left( \sum_{n=0}^\infty \frac{(-\frac{\parxi}{1-\parxi};\theta)_n}{(\theta;\theta)_n}\, \left((1-\theta)(1-\parxi)z \right)^{n} -1\right) - 1\nonumber
\\
&=&
\sum_{n=1}^\infty \frac{(-\frac{\parxi}{1-\parxi};\theta)_{n+1}}{(\theta;\theta)_{n+1}}\,  (1-\theta)^{n+1}(1-\parxi)^{n+1}  z^n. \nonumber
\end{eqnarray}
Now, let $\theta>1$. Then $p_n^{1/\theta,1-\parxi}=p_n^{\theta,\parxi}$, and the same holds for the right hand-side in the claim, as can be checked:
\begin{equation*}
    \frac{(-\frac{1-\parxi}{\parxi};\frac{1}{\theta})_{n+1}}{(\frac{1}{\theta};\frac{1}{\theta})_{n+1}}\, \left(1-\frac{1}{\theta}\right)^{n+1} \parxi^{n+1}
    =
    \frac{(-\frac{\parxi}{1-\parxi};\theta)_{n+1}}{(\theta;\theta)_{n+1}}\,(1-\theta)^{n+1} (1-\parxi)^{n+1},
\end{equation*}
%\begin{eqnarray*}
%\frac{(-\frac{1-\parxi}{\parxi};\frac{1}{\theta})_{n+1}}{(\frac{1}{\theta};\frac{1}{\theta})_{n+1}}\, (1-\frac{1}{\theta})^{n+1} \parxi^{n+1}
%&= &
%\prod_{k=0}^n \left[ \frac{1-\frac{1-\parxi}{\parxi}\theta^{-k}}{1+\theta^{-1} \theta^{-k}} (1-\theta^{-1}) \parxi \right]
%\\
%&=&
%\prod_{k=0}^n \left[ \frac{\theta^{-k} \frac{1-\parxi}{\parxi}(\frac{\parxi}{1-\parxi}\theta^k-1)}{\theta^{-k-1}(\theta \theta^k+1)}  \, \theta^{-1}(\theta-1) \parxi \right]
%\\
%&=&
%\prod_{k=0}^n \left[ \frac{ (1-\parxi)(\frac{\parxi}{1-\parxi} \theta^k-1)}{\theta \theta^k+1}  \, (\theta-1) \right]
%\\
%&=&
%\prod_{k=0}^n \left[ \frac{1-\frac{\parxi}{1-\parxi} \theta^k}{1+\theta \theta^k}  \, (1-\theta)(1-\parxi) \right]
%\\
%&=&
% \frac{(-\frac{\parxi}{1-\parxi};\theta)_{n+1}}{(\theta;\theta)_{n+1}}\,(1-\theta)^{n+1} (1-\parxi)^{n+1},
%\end{eqnarray*}
using the symmetry formula 
    $(a;\frac{1}{q})_n = (-a)^n q^{-\frac{n(n-1)}{2}}(\frac{1}{a};q)_n$.
Note that we could not treat the case $\theta>1$ directly, because the infinite $q$-Pochhammer symbols, e.g.\ $((1-\theta)z;\theta)_\infty$, are not properly defined for $\theta>1$.
\end{proof}

\section{Proofs for the case \texorpdfstring{$\theta<0$}{of negative theta}} \label{sec:negative}
We recall, for $\theta<0$, the borderline cases $\parxi=0$ and $\parxi=1$ are trivial because we have $p_n^{\theta,0}=1$ and $p_n^{\theta,1}=0$ for all $n$. The main result, Theorem~\ref{thm:gfthetaneg}, trivially holds for these cases. Thus, we will exclude them in the subsequent considerations.

We shall again use the auxiliary function $p_n(z)$ as defined in \eqref{eq:auxiliary_function}.
%$$
%p_n(z):=\P( X_1 > \theta z, X_2 > \theta X_1, \ldots, X_n > \theta X_{n-1} ),\qquad n=1,2,\ldots.
%$$
Clearly, for $\theta<0$, the function $z\mapsto p_n(z)$ is increasing with $p_n(-\infty)=0$ and $p_n(+\infty)=p_{n-1}$ for $n\geq 2$.

We start with explicit formulas for $p_n(z)$ for negative and positive $z$. These formulas involve sequences $(\alpha_n)$, $(r_i^n)$, $(s_i^n)$, and $(q_n)$, which are defined recursively below.

\begin{lem} Let $\theta<0$ and $\parxi\in(0,1)$. Set
$
\alpha_1:=\theta$ and for $n\geq 2$, $\alpha_n:=(1-\alpha_{n-1})\theta
$.
Then, for any $n\geq 1$,
\begin{equation}
\label{eqn:thetanetznegnasatz}
p_n(z) = \sum_{j=1}^n \frac{(-1)^{j-1}(1-\parxi)^{j-1} r_{n-j+1}^{n}}{\prod_{m=1}^{j-1} (1-\alpha_m)} e^{-\alpha_j z},\qquad z\leq 0,
\end{equation}
and
\begin{equation} \label{eqn:thetanetzposnasatz}
p_n(z) = \sum_{j=1}^n \frac{(-1)^{j-1}\parxi^{j-1}s_{n-j+1}^{n}}{\prod_{m=1}^{j-1} (1-\alpha_m)} (1-e^{\alpha_j z}) + q_n,\qquad z\geq 0,
\end{equation}
where $r_1^{1}:=1-\parxi$, $s_1^{1}:=\parxi$, $q_1:=1-\parxi$, while, for $n\geq 2$, we set $s_n^n:=0$ and 
\begin{equation} \label{eqn:recursionrvss}
r_{i}^n:=\left( \frac{\parxi}{1-\parxi}\right)^{n-i-1} s_{i}^{n-1}, \quad s_i^n:=(-1)\left( \frac{1-\parxi}{\parxi}\right)^{n-i-1} r_i^{n-1},\quad i=1,\ldots,n-1,
\end{equation}
\begin{equation} \label{eqn:recursionr+-}
r_n^n := \sum_{j=1}^{n-1} \frac{(-1)^{j-1} \parxi^{j-1} (1-\parxi) s_{n-j}^{n-1}}{\prod_{m=1}^{j-1} (1-\alpha_m)}+ q_{n-1}(1-\parxi)
\end{equation}
and
\begin{equation} \label{eqn:recursionqn}
q_n:= r_n^n- \sum_{j=1}^{n-1} \frac{(-1)^{j-1}\parxi^{j-1} (1-\parxi) s_{n-j}^{n-1}}{\prod_{m=1}^{j} (1-\alpha_m)}.
%r_n^n-\sum_{j=1}^{n-1} (-1)^{j-1} \, \frac{s_{n-j}^{n-1}}{2^{j} \prod_{m=1}^{j} (1-\alpha_m)}.
\end{equation}
\end{lem}

% We remark that the sequences $(\alpha_n)$, $(r_i^n)_{i=1,\ldots,n; n\geq 1}$, $(s_i^n)_{i=1,\ldots,n; n\geq 1}$, and $(q_n)$ are defined fully recursively in the lemma. 

\begin{proof}
The case $n=1$ follows immediately:
$$
p_1(z)=\int_{\theta z}^\infty e^{-x_1} (1-\parxi) \dd x_1 = (1-\parxi) e^{-\theta z},\qquad z\leq 0,
$$
while
$$
p_1(z)=\int_{\theta z}^0  e^{x_1}  \parxi \dd x_1+\int_0^\infty  e^{-x_1}  (1-\parxi) \dd x_1 = \parxi(1- e^{\theta z}) + 1-\parxi,\qquad z\geq 0.
$$
Thus, indeed the initial values are $\alpha_1=\theta$, $r_1^1=1-\parxi$, $s_1^1=\parxi$, and $q_1=1-\parxi$.

For the induction step, we need the following two main observations:
For $z\leq 0$, we have (since $\theta z\geq 0$)
\begin{equation} \label{eqn:pnobsthetanegzneg}
p_n(z) = \int_{\theta z}^\infty p_{n-1}(x_1) e^{-x_1}  (1-\parxi) \dd x_1 .
\end{equation}
Contrary, for $z\geq 0$, we have (since now $\theta z \leq 0$)

\begin{equation} \label{eqn:pnobsthetanegzpos}
p_n(z) = \int_{\theta z}^0 p_{n-1}(x_1) e^{x_1} \parxi \dd x_1  + p_n(0).
\end{equation}
% This means that the formula for $p_n(z)$ for positive $z$ involves $p_{n-1}(x_1)$ for $x_1$ negative, while the formula for $p_n(z)$ for $z$ negative involves $p_{n-1}(x_1)$ for positive $x_1$.

Assume now that  \eqref{eqn:thetanetznegnasatz} and \eqref{eqn:thetanetzposnasatz} hold for the case $n-1$, $n\geq 2$. Then using \eqref{eqn:pnobsthetanegzneg} and the induction hypothesis, we have, for $z\leq 0$,
\begin{eqnarray}
 p_n(z)&=& \int_{\theta z}^\infty p_{n-1}(x_1) e^{-x_1}  (1-\parxi) \dd x_1 
 \notag
 \\
 &=& \int_{\theta z}^\infty \left(\sum_{j=1}^{n-1}\frac{ (-1)^{j-1}\parxi^{j-1} s_{n-1-j+1}^{n-1}}{\prod_{m=1}^{j-1} (1-\alpha_m)} (1-e^{\alpha_j x_1}) + q_{n-1}\right) e^{-x_1}  (1-\parxi) \dd x_1 
 \notag
 \\
 &=& \sum_{j=1}^{n-1} \frac{(-1)^{j-1} \parxi^{j-1} s_{n-j}^{n-1}}{\prod_{m=1}^{j-1} (1-\alpha_m)} \int_{\theta z}^\infty ( e^{-x_1}- e^{-(1-\alpha_j) x_1} )(1-\parxi) \dd x_1  + q_{n-1}\int_{\theta z}^\infty  e^{-x_1} (1-\parxi) \dd x_1 
 \notag
 \\
 &=& \sum_{j=1}^{n-1} \frac{(-1)^{j-1} \parxi^{j-1} (1-\parxi) s_{n-j}^{n-1}}{\prod_{m=1}^{j-1} (1-\alpha_m)} \left( e^{-\theta z} - \frac{1}{1-\alpha_j} \, e^{-(1-\alpha_{j})\theta z} \right) +  q_{n-1} (1-\parxi)  e^{-\theta z}
 \notag
 \\
 &=& \sum_{j=1}^{n-1} \frac{(-1)^{j-1} \parxi^{j-1} (1-\parxi) s_{n-j}^{n-1}}{\prod_{m=1}^{j-1} (1-\alpha_m)} \cdot e^{-\alpha_1 z}
 \notag
 \\
 && + \sum_{j=1}^{n-1} \frac{(-1)^{j+1-1}\parxi^{j-1} (1-\parxi) s_{n-j}^{n-1}}{\prod_{m=1}^{j} (1-\alpha_m)} e^{-\alpha_{j+1} z}  +  q_{n-1}(1-\parxi) e^{-\alpha_1 z}
 \notag
 \\
 &=& \left(\sum_{j=1}^{n-1} \frac{(-1)^{j-1}\parxi^{j-1} (1-\parxi) s_{n-j}^{n-1}}{\prod_{m=1}^{j-1} (1-\alpha_m)}+q_{n-1} (1-\parxi)\right) \cdot e^{-\alpha_1 z}
 \notag
 \\
 && + \sum_{j=2}^{n} \frac{(-1)^{j-1}\parxi^{j-2} (1-\parxi) s_{n-j+1}^{n-1}}{\prod_{m=1}^{j-1} (1-\alpha_m)} e^{-\alpha_{j} z},
 \label{eqn:continuityatzero}
\end{eqnarray}
as required in \eqref{eqn:thetanetznegnasatz}, because $\parxi^{j-2} (1-\parxi) s_{n-j+1}^{n-1}=(1-\parxi)^{j-1} r_{n-j+1}^n$, for $j=2,\ldots,n$, and since \eqref{eqn:recursionr+-} makes the prefactor of $e^{-\alpha_1 z}$ above coincide with the one in \eqref{eqn:thetanetznegnasatz}. On the other hand, for $z\geq 0$, we get from \enlargethispage{\baselineskip} \eqref{eqn:pnobsthetanegzpos} and the induction hypothesis that
\begin{eqnarray*}
p_n(z) &=& 
\int_{\theta z}^0 p_{n-1}(x_1) e^{x_1} \parxi \dd x_1 + p_n(0)
\\
&=& 
\int_{\theta z}^0 \sum_{j=1}^{n-1}  \frac{(-1)^{j-1} (1-\parxi)^{j-1} r_{n-1-j+1}^{n-1}}{\prod_{m=1}^{j-1} (1-\alpha_m)}\, e^{-\alpha_j x_1}  e^{x_1} \parxi \dd x_1 + p_n(0)
\\
&=& 
\sum_{j=1}^{n-1}\frac{ (-1)^{j-1}(1-\parxi)^{j-1} r_{n-j}^{n-1}}{\prod_{m=1}^{j-1} (1-\alpha_m)}\int_{\theta z}^0 e^{(1-\alpha_j) x_1}\parxi \dd x_1+ p_n(0)
\\
&=& 
\sum_{j=1}^{n-1} \frac{(-1)^{j-1} (1-\parxi)^{j-1} \parxi r_{n-j}^{n-1}}{\prod_{m=1}^{j} (1-\alpha_m)}(1-e^{(1-\alpha_j) \theta z}) + p_n(0)
\\
&=& 
\sum_{j=1}^{n-1} \frac{(-1)^{j-1}(1-\parxi)^{j-1} \parxi r_{n-j}^{n-1}}{ \prod_{m=1}^{j} (1-\alpha_m)}(1-e^{\alpha_{j+1} z}) + p_n(0)
\\
&=& 
\sum_{j=2}^{n} \frac{(-1)^{j-1}(-1)(1-\parxi)^{j-2} \parxi  r_{n-j+1}^{n-1}}{\prod_{m=1}^{j-1} (1-\alpha_m)}(1-e^{\alpha_{j} z}) + p_n(0),
\end{eqnarray*}
as required in \eqref{eqn:thetanetzposnasatz}, because $-(1-\parxi)^{j-2} \parxi  r_{n-j+1}^{n-1}=\parxi^{j-1} s_{n-j+1}^n$ for $j=2,\ldots,n$, $s_n^n=0$, and 
\begin{eqnarray*} \label{eqn:recursionqs}
p_n(0) &=& r_n^n+\sum_{j=2}^n  \frac{(-1)^{j-1}\xi^{j-2}(1-\parxi)s_{n-j+1}^{n-1}}{\prod_{m=1}^{j-1} (1-\alpha_m)} 
\\
&=& r_n^n- \sum_{j=1}^{n-1} \frac{(-1)^{j-1}\parxi^{j-1} (1-\parxi) s_{n-j}^{n-1}}{\prod_{m=1}^{j} (1-\alpha_m)}
\\
&=&q_n,
\end{eqnarray*}
using \eqref{eqn:continuityatzero} in the first and \eqref{eqn:recursionqn} in the last step.
\end{proof}

The complicated recursions from the last lemma can be interpreted in terms of generating functions in order to show Theorem~\ref{thm:gfthetaneg}\ref{thm2-ith}.

\begin{proof}[Proof of Theorem~\ref{thm:gfthetaneg}\ref{thm2-ith}] Using $p_n(+\infty)=p_{n-1}$ for $n\geq 2$ as well as \eqref{eqn:thetanetzposnasatz}, we obtain
\begin{eqnarray}
    zP(z)&=&\sum_{n=2}^\infty p_{n-1} z^n 
    \notag \\
    &=& \sum_{n=2}^\infty \left(\sum_{j=1}^n  \frac{(-1)^{j-1} \parxi^{j-1} s_{n-j+1}^n}{ \prod_{m=1}^{j-1} (1-\alpha_m)} + q_n \right) z^n
    \notag \\
     &=& \sum_{n=1}^\infty \sum_{j=1}^n  \frac{(-1)^{j-1}\parxi^{j-1} s_{n-j+1}^n}{\prod_{m=1}^{j-1} (1-\alpha_m)} z^n - s_1^1 \cdot z+ \sum_{n=1}^\infty q_n z^n - q_1 z
    \notag \\
     &=& \sum_{j=1}^\infty \frac{ (-1)^{j-1}\parxi^{j-1}}{\prod_{m=1}^{j-1} (1-\alpha_m)} z^{j-1} \cdot \sum_{n=j}^\infty s_{n-j+1}^n z^{n-j+1} 
         \notag \\
     && - \parxi z + \sum_{n=1}^\infty q_n z^n - (1-\parxi) z
    \notag \\
     &=& \sum_{j=1}^\infty \frac{(-1)^{j-1}\parxi^{j-1}}{\prod_{m=1}^{j-1} (1-\alpha_m)} \,z^{j-1} \cdot S^j(z) + Q(z) -  z, \label{eqn:represpnega}
\end{eqnarray}
where
$$
S^j(z):=\sum_{\ell=1}^\infty s_\ell^{\ell+j-1}  z^\ell,\qquad j=1,2,\ldots,\qquad \text{and}\qquad Q(z):=\sum_{n=1}^\infty q_n z^n.
$$
We define in an analogous way $R^{j}(z):=\sum_{\ell=1}^\infty r_\ell^{\ell+j-1} z^\ell$. Set $\gamma:=\frac{1-\parxi}{\parxi}$. Note that, by (\ref{eqn:recursionrvss}), for $j=2,3,\ldots$
$$
S^j(z)=\sum_{\ell=1}^\infty s_\ell^{\ell+j-1}  z^\ell = \sum_{\ell=1}^\infty (-\gamma^{j-2} r_\ell^{\ell+j-1-1})  z^\ell = - \gamma^{j-2}\sum_{\ell=1}^\infty r_\ell^{\ell+(j-1)-1} z^\ell = -\gamma^{j-2}R^{j-1}(z).
$$
Analogously, by (\ref{eqn:recursionrvss}), for $j=2,3,\ldots$,
$$
R^j(z)=\sum_{\ell=1}^\infty r_\ell^{\ell+j-1}  z^\ell = \sum_{\ell=1}^\infty \gamma^{-(j-2)} s_\ell^{\ell+j-1-1}  z^\ell  = \gamma^{-(j-2)} S^{j-1}(z).
$$
Putting these two observations together gives $S^j(z)=-\gamma S^{j-2}(z)$ for $j\geq 3$. Further,
$$
S^1(z)=\sum_{\ell=1}^\infty s_\ell^{\ell}  z^\ell = s_1^1 z = \parxi \, z.
$$
This shows, by induction, that
% $$
% S^j(z)=\begin{cases}
%     (-1)^{j/2}\cdot \gamma^{\frac{j}{2}(\frac{j}{2}-1)} \bar\gamma^{(\frac{j}{2}-1)^2} \cdot R^1(z) & \text{$j$ even},\\
%     (-1)^{(j-1)/2}\cdot\gamma^{(\frac{j-1}{2})^2} \bar\gamma^{\frac{j-1}{2} (\frac{j-1}{2}-1)} \cdot S^1(z) & \text{$j$ odd}
% \end{cases}
% $$
% which can be simplified to
$$
S^j(z)=\begin{cases}
    (-1)^{j/2}\cdot \gamma^{\frac{j-2}{2}} \cdot R^1(z), &\quad\text{$j\geq 2$ even},\\
    (-1)^{(j-1)/2}\cdot \gamma^{\frac{j-1}{2}}\cdot\parxi z, &\quad \text{$j\geq 1$ odd}.
\end{cases}
$$
We can now use \eqref{eqn:recursionr+-} to express $R^1$ in terms of $Q$ and the functions
\begin{equation*}
    g_1(z):=\sin_{-\theta}(z) = \sum_{k=0}^\infty \frac{(-1)^k z^{2k+1}}{\prod_{m=0}^{2k} (1-\alpha_m)}\quad\text{and}\quad
        g_2(z):=\cos_{-\theta}(z) = \sum_{k=0}^\infty \frac{(-1)^k z^{2k}}{\prod_{m=0}^{2k-1} (1-\alpha_m)},
\end{equation*}
where we used that $\prod_{i=0}^n (1-\alpha_i) = \frac{(-\theta,-\theta)_{n+1}}{(1-(-\theta))^{n+1}}=[n+1]_{-\theta}!$ and \eqref{eqn:qversionstrigexp}.
% \begin{eqnarray*}
% g_1(z) &:=& \sum_{k=1}^\infty \frac{1}{2^{2k} \prod_{m=1}^{2k-1} (1-\alpha_m)} z^{2k} (-1)^k, 
% \\
% h_1(z)&:=& \sum_{k=0}^\infty  \frac{1}{2^{2k+1} \prod_{m=1}^{2k} (1-\alpha_m)} z^{2k+1} (-1)^k \frac{z}{2}=\frac{z^2}{4}( g_2(z)+1),
% \\
% g_2(z)&:=&\sum_{k=1}^\infty \frac{1}{2^{2k} \prod_{m=1}^{2k} (1-\alpha_m)} z^{2k} (-1)^k,
% \\
% h_2(z)&:=& \sum_{k=0}^\infty  \frac{1}{2^{2k+1} \prod_{m=1}^{2k+1} (1-\alpha_m)} z^{2k+1} (-1)^k \frac{z}{2} =-g_1(z).
% \end{eqnarray*}
Indeed, \eqref{eqn:recursionr+-} implies
\begin{eqnarray*}
R^1(z) &=& r_1^1 z + \sum_{\ell=2}^\infty r_\ell^\ell z^\ell = (1-\parxi) z+\sum_{\ell=2}^\infty \sum_{j=1}^{\ell-1}  \frac{(-1)^{j-1}\parxi^{j-1} (1-\parxi)s_{\ell-j}^{\ell-1}}{ \prod_{m=1}^{j-1} (1-\alpha_m)}z^\ell + \sum_{\ell=2}^\infty q_{\ell-1}(1-\parxi) z^\ell
\\
 &=& (1-\parxi) z+\sum_{j=1}^\infty \frac{(-1)^{j-1}\parxi^{j-1}(1-\parxi)}{ \prod_{m=1}^{j-1} (1-\alpha_m)} z^j \sum_{\ell=j+1}^\infty s_{\ell-j}^{\ell-1} z^{\ell-j} + \sum_{\ell=2}^\infty q_{\ell-1}(1-\parxi) z^{\ell-1} \, z
 \\
 &=& (1-\parxi) z+\sum_{j=1}^\infty \frac{(-1)^{j-1}\parxi^{j-1}(1-\parxi)}{\prod_{m=1}^{j-1} (1-\alpha_m)} z^j \sum_{\ell=1}^\infty s_{\ell}^{\ell+j-1} z^{\ell} +  Q(z) (1-\parxi) z
 \\
  &=& (1-\parxi) z+\sum_{j=1}^\infty \frac{(-1)^{j-1}\parxi^{j-1}(1-\parxi)}{ \prod_{m=1}^{j-1} (1-\alpha_m)} z^j S^j(z) +   Q(z) (1-\parxi) z
   \\
  % &=& (1-\parxi) z+\sum_{j=1,j\, \text{even}}^\infty\frac{(-1)^{j-1}\parxi^{j-1}(1-\parxi)}{ \prod_{m=1}^{j-1} (1-\alpha_m)} z^j S^j(z)
  % \\
  % && +\sum_{j=1,j\,\text{odd}}^\infty\frac{(-1)^{j-1}\parxi^{j-1}(1-\parxi)}{ \prod_{m=1}^{j-1} (1-\alpha_m)} z^j S^j(z) +  Q(z) (1-\parxi) z
  %    \\
  &=& (1-\parxi) z-(1-\parxi)\left(\sum_{k=1}^\infty \frac{\parxi^{2k-1}\gamma^{k-1}}{\prod_{m=1}^{2k-1} (1-\alpha_m)}\, z^{2k} (-1)^k\right) \cdot R^1(z)
  \\
  && +(1-\parxi)\sum_{k=0}^\infty  \frac{\parxi^{2k}\gamma^{k}}{\prod_{m=1}^{2k} (1-\alpha_m)} z^{2k+1} (-1)^k \parxi z +   Q(z) (1-\parxi) z
       \\
  &=& (1-\parxi) z-(1-\parxi)\parxi^{-1}\gamma^{-1} \left(\sum_{k=1}^\infty \frac{(-1)^k(\parxi \sqrt{\gamma} z)^{2k}}{\prod_{m=1}^{2k-1} (1-\alpha_m)} \right) \cdot R^1(z)
  \\
  && +(1-\parxi)z\gamma^{-1/2}\sum_{k=0}^\infty  \frac{(-1)^k(\parxi \sqrt{\gamma} z)^{2k+1}}{\prod_{m=1}^{2k} (1-\alpha_m)} +   Q(z) (1-\parxi) z
         \\
  &=& (1-\parxi) z-(g_2(\parxi \sqrt{\gamma} z)-1) \cdot R^1(z)
   +(1-\parxi)z\gamma^{-1/2} g_1(\parxi \sqrt{\gamma} z) +   Q(z) (1-\parxi) z.
\end{eqnarray*}
This shows that
\begin{equation} \label{eqn:replaced}
R^1(z)
  = \frac{1}{g_2(\parxi \sqrt{\gamma} z)}\left( (1-\parxi) z 
   +(1-\parxi)z\gamma^{-1/2} g_1(\parxi \sqrt{\gamma} z) +   Q(z) (1-\parxi) z\right).
\end{equation}
This is an equation for $R^1(z)$ and $Q(z)$. We now use \eqref{eqn:recursionqn} in order to find a complementary equation:
\begin{eqnarray*}
Q(z)&=& \sum_{n=1}^\infty q_n z^n = \sum_{n=2}^\infty q_n z^n +q_1 z 
\\
&=& \sum_{n=2}^\infty r_n^n z^n -\sum_{n=2}^\infty \sum_{j=1}^{n-1} \frac{(-1)^{j-1} \parxi^{j-1} (1-\parxi) s_{n-j}^{n-1}}{ \prod_{m=1}^{j} (1-\alpha_m)} z^n+(1-\parxi)z
\\
&=& \sum_{n=1}^\infty r_n^n z^n -(1-\parxi) z  -\sum_{j=1}^\infty\frac{(-1)^{j-1} \parxi^{j-1}(1-\parxi)}{\prod_{m=1}^{j} (1-\alpha_m)} z^j \sum_{n=j+1}^\infty s_{n-j}^{n-1} z^{n-j}+(1-\parxi)z
\\
&=& \sum_{n=1}^\infty r_n^n z^n   -\sum_{j=1}^\infty \frac{(-1)^{j-1}\parxi^{j-1} (1-\parxi)}{\prod_{m=1}^{j} (1-\alpha_m)} z^j \sum_{\ell=1}^\infty s_{\ell}^{\ell+j-1} z^{\ell}
\\
&=& R^1(z)   -\sum_{j=1}^\infty \frac{ (-1)^{j-1}\parxi^{j-1}(1-\parxi)}{ \prod_{m=1}^{j} (1-\alpha_m)} z^j S^j(z)
\\
% &=& R^1(z)   -\sum_{j=1, j\,\text{even}}^\infty \frac{ (-1)^{j-1}\parxi^{j-1}(1-\parxi)}{ \prod_{m=1}^{j} (1-\alpha_m)} z^j S^j(z)
% \\
% &&-\sum_{j=1, j\,\text{odd}}^\infty \frac{ (-1)^{j-1}\parxi^{j-1}(1-\parxi)}{ \prod_{m=1}^{j} (1-\alpha_m)} z^j S^j(z)
% \\
&=& R^1(z)  +(1-\parxi)\left( \sum_{k=1}^\infty \frac{\parxi^{2k-1}\gamma^{k-1}}{\prod_{m=1}^{2k} (1-\alpha_m)} z^{2k} (-1)^k\right) R^1(z)
\\
&&  -(1-\parxi)\sum_{k=0}^\infty  \frac{\parxi^{2k}\gamma^{k}}{ \prod_{m=1}^{2k+1} (1-\alpha_m)} z^{2k+1} (-1)^k  \parxi z
\\
&=& R^1(z)  +(1-\parxi)\parxi^{-2}\gamma^{-3/2}z^{-1} \left( \sum_{k=1}^\infty \frac{(-1)^k (\parxi \sqrt{\gamma} z)^{2k+1}}{\prod_{m=1}^{2k} (1-\alpha_m)} \right) R^1(z)
\\
&&  +(1-\parxi)\sum_{k=1}^\infty  \frac{\parxi^{2k-2}\gamma^{k-1}}{ \prod_{m=1}^{2k-1} (1-\alpha_m)} z^{2k-1} (-1)^k  \parxi z
\\
&=& R^1(z)  +(1-\parxi)\parxi^{-2}\gamma^{-3/2}z^{-1} \left( \sum_{k=1}^\infty \frac{(-1)^k (\parxi \sqrt{\gamma} z)^{2k+1}}{\prod_{m=1}^{2k} (1-\alpha_m)} \right) R^1(z)
\\
&&  +(1-\parxi)\parxi^{-1}\gamma^{-1}\sum_{k=1}^\infty  \frac{(-1)^k (\parxi\sqrt{\gamma}z)^{2k}}{ \prod_{m=1}^{2k-1} (1-\alpha_m)}
\\
&=& R^1(z) +\parxi^{-1}\gamma^{-1/2}z^{-1} (g_1(\parxi \sqrt{\gamma}z) -\parxi\gamma^{1/2}z)R^1(z) +g_2(\parxi\sqrt{\gamma}z)-1
\\
&=&  \parxi^{-1}\gamma^{-1/2}z^{-1} g_1(\parxi \sqrt{\gamma}z) R^1(z) +g_2(\parxi\sqrt{\gamma}z)-1.
\end{eqnarray*}
Combining with \eqref{eqn:replaced}, we obtain
\begin{eqnarray*}
R^1(z)
  &=& \frac{1}{g_2(\parxi \sqrt{\gamma} z)}\bigl( (1-\parxi) z 
   +(1-\parxi)z\gamma^{-1/2} g_1(\parxi \sqrt{\gamma} z)
   \\
   && +  (1-\parxi) z\parxi^{-1}\gamma^{-1/2}z^{-1} g_1(\parxi \sqrt{\gamma}z) R^1(z) +(1-\parxi) z g_2(\parxi\sqrt{\gamma}z)-(1-\parxi) z\bigr)
   \\
     &=& \frac{1}{g_2(\parxi \sqrt{\gamma} z)}\bigl( (1-\parxi)z\gamma^{-1/2} g_1(\parxi \sqrt{\gamma} z)
   \\
   && + \gamma^{1/2} g_1(\parxi \sqrt{\gamma}z) R^1(z) +(1-\parxi) z g_2(\parxi\sqrt{\gamma}z)\bigr),
\end{eqnarray*}
so that
\begin{equation*}
\frac{g_2(\parxi \sqrt{\gamma} z)-\gamma^{1/2} g_1(\parxi \sqrt{\gamma}z)}{g_2(\parxi \sqrt{\gamma} z)}\, R^1(z)
     = \frac{(1-\parxi) z}{g_2(\parxi \sqrt{\gamma} z)}\left( \gamma^{-1/2} g_1(\parxi \sqrt{\gamma} z) +  g_2(\parxi\sqrt{\gamma}z)\right),
\end{equation*}
which implies
\begin{equation*}
R^1(z)
     = \frac{(1-\parxi) z\left(g_2(\parxi\sqrt{\gamma}z)+\gamma^{-1/2} g_1(\parxi \sqrt{\gamma} z)\right)}{g_2(\parxi \sqrt{\gamma} z)-\gamma^{1/2} g_1(\parxi \sqrt{\gamma}z)}.
\end{equation*}
This gives
\begin{eqnarray*}
Q(z)
&=& \parxi^{-1}\gamma^{-1/2}z^{-1} g_1(\parxi \sqrt{\gamma}z) R^1(z) +g_2(\parxi\sqrt{\gamma}z)-1
\\
&=& \parxi^{-1}\gamma^{-1/2}z^{-1} g_1(\parxi \sqrt{\gamma}z) \frac{(1-\parxi) z\left(g_2(\parxi\sqrt{\gamma}z)+\gamma^{-1/2} g_1(\parxi \sqrt{\gamma} z)\right)}{g_2(\parxi \sqrt{\gamma} z)-\gamma^{1/2} g_1(\parxi \sqrt{\gamma}z)} +g_2(\parxi\sqrt{\gamma}z)-1.
\end{eqnarray*}

This allows us to continue in \eqref{eqn:represpnega}. For better readability, we are going to use the abbreviations $\gge:=g_1(\parxi \sqrt{\gamma}z)$ and $\ggz:=g_2(\parxi \sqrt{\gamma}z)$. We obtain:
\begin{eqnarray*}
    P(z)
     &=&\frac{1}{z}\left(  \sum_{j=1}^\infty \frac{(-1)^{j-1}\parxi^{j-1}}{\prod_{m=1}^{j-1} (1-\alpha_m)} z^{j-1} \cdot S^j(z) + Q(z) -  z\right)
     \\
     % &=&\frac{1}{z}\left(  \sum_{j=1, j\,\text{even}}^\infty\frac{(-1)^{j-1}\parxi^{j-1}}{\prod_{m=1}^{j-1} (1-\alpha_m)} z^{j-1} \cdot S^j(z)\right.
     % \\
     % && \left. + \sum_{j=1, j\,\text{odd}}^\infty\frac{(-1)^{j-1}\parxi^{j-1}}{\prod_{m=1}^{j-1} (1-\alpha_m)} z^{j-1} \cdot S^j(z) + Q(z) -  z\right)
     %      \\
     &=&\frac{1}{z}\left(  - \sum_{k=1}^\infty \frac{\parxi^{2k-1} \gamma^{k-1}}{\prod_{m=1}^{2k-1} (1-\alpha_m)} z^{2k-1} (-1)^k \cdot R^1(z)\right.
     \\
     && \left. + \sum_{k=0}^\infty \frac{\parxi^{2k}\gamma^{k}}{\prod_{m=1}^{2k} (1-\alpha_m)} z^{2k}(-1)^{k} \cdot \parxi z + Q(z) -  z\right)
               \\
     &=&\frac{1}{z}\left(  - \parxi^{-1}\gamma^{-1} z^{-1}\sum_{k=1}^\infty \frac{(-1)^k(\parxi\sqrt{\gamma}z)^{2k}}{\prod_{m=1}^{2k-1} (1-\alpha_m)}  \cdot R^1(z)\right.
     \\
     && \left. + \gamma^{-1/2}\sum_{k=0}^\infty \frac{(-1)^{k}(\parxi\sqrt{\gamma}z)^{2k+1}}{\prod_{m=1}^{2k} (1-\alpha_m)} + Q(z) -  z\right)
               \\
     &=&
     \frac{1}{z}\left(-\parxi^{-1}\gamma^{-1} z^{-1} (g_2(\parxi\sqrt{\gamma}z)-1)R^1(z) +\gamma^{-1/2}g_1(\parxi\sqrt{\gamma}z)+Q(z) -z \right)
     \\
          &=&
     \frac{1}{z}\left[-\parxi^{-1}\gamma^{-1} z^{-1} (\ggz-1)\left( \frac{(1-\parxi) z(\ggz+\gamma^{-1/2} \gge)}{\ggz-\gamma^{1/2} \gge}\right)\right.
     \\
     && \left.+\gamma^{-1/2}\gge+ \parxi^{-1}\gamma^{-1/2}z^{-1} \gge \frac{(1-\parxi) z(\ggz+\gamma^{-1/2}\gge )}{\ggz-\gamma^{1/2} \gge} +\ggz-1\right]-1
          \\
          &=&
     \frac{1}{z(\ggz-\gamma^{1/2} \gge)}\left[-  \ggz^2 +\ggz -\ggz\gamma^{-1/2}\gge+\gamma^{-1/2}\gge+\gamma^{-1/2}\gge\ggz-\gamma^{-1/2}\gge\gamma^{1/2} \gge\right.
     \\
     &&+ \left.\gamma^{1/2} \gge \ggz+ \gge^2+\ggz^2 - \ggz - \gamma^{1/2} \gge \ggz +\gamma^{1/2} \gge\right]-1
          \\
          &=&
     \frac{1}{z(\ggz-\gamma^{1/2} \gge)}\left[\gamma^{-1/2} \gge +\gamma^{1/2} \gge\right]-1
               \\
          &=&
     \frac{\frac{1}{\sqrt{\parxi(1-\parxi)}}\,\gge}{z(\ggz-\gamma^{1/2} \gge)}-1,
\end{eqnarray*}
where we used in the last step that $\gamma^{-1/2} +\gamma^{1/2}=\gamma^{1/2}(\frac{\parxi}{1-\parxi}+1)=\frac{1}{\sqrt{\parxi(1-\parxi)}}$. We further note that $\parxi \sqrt{\gamma}=\sqrt{\parxi (1-\parxi)}$. Replacing these values, we can conclude. 
\end{proof}

% \section{An identity for Stanley's $q$-analogues of the zigzag numbers}

The purpose of next lemma is to represent the generating function in \eqref{eq:main_formula_theta<0} as a series with coefficients defined in terms of Stanley's $q$-analogues of the zigzag numbers.
Let us remind the reader of the definitions of the $q$-versions of sine and cosine from \eqref{eqn:qversionstrigexp} and the definitions of Stanley's $q$-analogues of the Euler zigzag numbers, $E_n(q)$, from \eqref{eq:Stanley_q_analogue}. 

\begin{lem} \label{lem:chatgptlemma} Let $q>0$ and $\delta\in[0,1]$. Let $|z|<z_0$, where $z_0$ is the radius of convergence of $\tan_q:=\frac{\sin_q}{\cos_q}$. Then 
$$
\frac{\sin_q(z)}{\cos_q(z)-\delta \sin_q(z)} = \sum_{n=1}^\infty b_n(q) z^n,
$$
where
\begin{equation} \label{eqn:bnlemm10}
b_n(q) := \sum_{r=1}^{n} \left[\delta^{r-1}
  \ \sum_{\substack{j_1+\dots+j_r=n\\ j_k\ \text{odd}}}
  \prod_{k=1}^r \frac{E_{j_k}(q)}{[j_k]_q!}\right],
\qquad n\geq 1.
\end{equation}
\end{lem}

\begin{proof}
Recall the equation \eqref{eq:Stanley_q_analogue}.
Since $\cos_q(z)$ is even and $\sin_q(z)$ is odd, we obtain the even/odd
splitting
\[
\frac{1}{\cos_q(z)}
= \sum_{m=0}^\infty E_{2m}(q)\,\frac{z^{2m}}{[2m]_q!},
\qquad
\tan_q(z):=\frac{\sin_q(z)}{\cos_q(z)}
= \sum_{m=0}^\infty E_{2m+1}(q)\,\frac{z^{2m+1}}{[2m+1]_q!}.
\]

% Observe that, since $\cos_q(z)\neq 0$,
% \[
% \cos_q(z) - \delta \sin_q(z)
% = \cos_q(z)\Bigl(1 - \delta \tan_q(z)\Bigr),
% \]
% so that
Since $\cos_q(z)\neq 0$, we have
\[
\frac{\sin_q(z)}{\cos_q(z)-\sin_q(z)}
= \frac{\tan_q(z)\cos_q(z)}{\cos_q(z)(1-\delta \tan_q(z))}
= \frac{\tan_q(z)}{1- \delta \tan_q(z)}.
\]
Let us rewrite the series for $\tan_q(z)$ using the odd $q$-Euler numbers:
\[
\tan_q(z) = \sum_{j=1}^\infty
  u_j(q)\,z^j,
\qquad 
\text{where}
\qquad
u_j(q) :=
\begin{cases}
\dfrac{E_j(q)}{[j]_q!}, & j\ \text{odd},\\[4pt]
0, & j\ \text{even}.
\end{cases}
\]

Now we expand using the geometric series:
\begin{eqnarray*}
    \frac{\sin_q(z)}{\cos_q(z)-\delta \sin_q(z)}
    &=& \frac{1}{\delta}\, \frac{\delta \tan_q(z)}{1-\delta \tan_q(z)}
    \\
    &=&
\frac{1}{\delta}\,\sum_{r=1}^\infty \delta^r \tan_q(z)^r
\\
&=&
\frac{1}{\delta}\,\sum_{r=1}^\infty \delta^r \sum_{j_1,\ldots,j_r=1}^\infty u_{j_1}(q)\cdot\ldots\cdot u_{j_r}(q) z^{j_1+\ldots+j_r}
\\
&=& \frac{1}{\delta}\, \sum_{r=1}^\infty \delta^r
    \sum_{n=r}^\infty
    \sum_{j_1+\cdots+j_r=n}
      u_{j_1}(q)\cdot\ldots\cdot u_{j_r}(q)\,z^n.
\end{eqnarray*}
Then the coefficient of $z^n$ is
\[
b_n(q)
= \frac{1}{\delta}\, \sum_{r=1}^{n} \left[ \delta^r
  \sum_{j_1+\cdots+j_r=n}
  u_{j_1}(q)\cdots u_{j_r}(q)\right].
\]
Since $u_j(q)=0$ for even $j$, we may restrict to the odd parts and obtain (\ref{eqn:bnlemm10}).
% \[
% \begin{aligned}
% b_n(q)
% &= \sum_{r=1}^{n}\left[ \delta^{r-1}
%      \sum_{\substack{j_1+\cdots+j_r=n\\ j_k\ \text{odd}}}
%      \prod_{k=1}^r \frac{E_{j_k}(q)}{[j_k]_q!}\right],
% % &= \sum_{r=1}^{n}
% %      \sum_{\substack{j_1+\cdots+j_r=n\\ j_k\ \text{odd}}}
% %      \prod_{k=1}^r \frac{E_{j_k}(q)}{[j_k]_q!},
% \end{aligned}
% \]
% as claimed.
\end{proof}
% Thus we have the explicit formula
% \[
% \boxed{
% b_n(q)
% = 2^{\,1-n}
%   \sum_{r=1}^{n}
%   \ \sum_{\substack{j_1+\dots+j_r=n\\ j_k\ \text{odd}}}
%   \prod_{k=1}^r \frac{E_{j_k}(q)}{[j_k]_q!}
% \qquad(n\ge1).
% }
% \]

% Equivalently, using the $q$-multinomial coefficients
% \[
% \binom{n}{j_1,\dots,j_r}_q
% := \frac{[n]_q!}{[j_1]_q!\cdots [j_r]_q!},
% \]
% we can write
% \[
% \boxed{
% b_n(q)
% = \frac{2^{\,1-n}}{[n]_q!}
%   \sum_{r=1}^{n}
%   \ \sum_{\substack{j_1+\dots+j_r=n\\ j_k\ \text{odd}}}
%   \binom{n}{j_1,\dots,j_r}_q
%   \prod_{k=1}^r E_{j_k}(q)
% \qquad(n\ge1).
% }
% \]

\begin{proof}[Proof of part \ref{thm2-iti} of Theorem~\ref{thm:gfthetaneg}]
    We put together formula \eqref{eq:main_formula_theta<0} with Lemma~\ref{lem:chatgptlemma}. Here, $\delta:=\sqrt{(1-\parxi)/\parxi}$ and $q:=-\theta$ and we use that $b_1(q)=1$.
\end{proof}

\begin{lem}
\label{lem:polynomial_in_xi}    
The quantity $p_n^{-1,\parxi}$ in \eqref{eqn:majumdardharextension} is a polynomial in $\parxi$.
\end{lem}

\begin{proof}
It follows from \eqref{eqn:majumdardharextension} and some computation that the sequence $(p_n^{-1,\parxi})$ satisfies the recursion
\begin{equation*}
    (n + 2)p_{n+1}^{-1,\parxi} = \bigl(n(1 - 2\parxi) + 2(1 - \parxi)\bigr)p_{n}^{-1,\parxi}- 2\parxi(1 - \parxi)\frac{d}{d\parxi}p_{n}^{-1,\parxi}.
\end{equation*}
Alternatively, it admits the following formula: define the polynomial $Q_n$ by $\bigl(\frac{1}{\tan}\bigr)^{(n)} = Q_n\bigl(\frac{1}{\tan}\bigr)$. We have $Q_0(z)=z$, $Q_1(z)=-1-z^2$, $Q_2(z)=2z+2z^3$, etc. The triangular array of coefficients appears here: \href{https://oeis.org/A008293}{A008293}. Then 
\begin{equation*}
    p_{n}^{-1,\parxi} =\frac{-1}{(n+1)!} \parxi \bigl(\parxi(1-\parxi)\bigr)^{n/2} Q_{n+1}\left(\sqrt{\frac{1-\parxi}{\parxi}}\right). \qedhere
\end{equation*}
\end{proof}

To conclude, we present the proof of Corollary~\ref{cor:asymptotics_theta<0}. As a first step, we derive an explicit expression for the function
% \begin{equation*}
    $\tan_q(z) := \frac{\sin_q(z)}{\cos_q(z)}$,
%\end{equation*}
where the $q$-sine and $q$-cosine functions are introduced in \eqref{eqn:qversionstrigexp}.
\begin{lem}
\label{lem:tan_q}
Let $0<q<1$. Let $z_0=z_0(q)$ be the smallest positive real zero of $\cos_q$. Then, for real $z\in(-z_0,z_0)$, we have 
    \begin{equation*}
    \tan_q(z)
=  \tan\left( \sum_{j=0}^\infty \arctan \bigl( (1-q)zq^{j})\bigr) \right).
\end{equation*}
\end{lem}
Clearly, $\tan_q$ has well-defined limits as $q\to0$ and $q\to1$, namely the identity function in the former case and the classical tangent in the latter.
\begin{proof}
By the $q$-binomial theorem,
$$
e_q(z) = \sum_{n=0}^\infty \frac{z^{n}}{[n]_q!}= \sum_{n=0}^\infty \frac{(1-q)^n z^{n}}{(q;q)_n} = \frac{1}{((1-q)z;q)_\infty} = \prod_{j=0}^\infty \frac{1}{1-(1-q)zq^j}.
$$

Recall that
$$
\sin_q(z)=\frac{1}{2i}( e_q(iz)-e_q(-iz))\qquad\text{and}\qquad\cos_q(z)=\frac{1}{2}( e_q(iz)+e_q(-iz)),
$$
which  can be seen as in the classical case.
Now note that
$$
e_q(iz) =  \prod_{j=0}^\infty \frac{1}{1-i(1-q)zq^j} = \prod_{j=0}^\infty \frac{1+i(1-q)zq^j}{1+(1-q)^2z^2q^{2j}}.
$$
Each factor can be written in polar form
$$
\frac{1+i(1-q)zq^j}{1+(1-q)^2z^2q^{2j}} = r_j e^{a_j i} =  r_j(\cos(a_j) + i \sin(a_j));
$$
and $(1-q)zq^j=\frac{\sin(a_j)}{\cos(a_j)}$ yields
$
a_j=\arctan( (1-q)zq^j )
$.
From here we have
$$
e_q(iz) =\prod_{j=0}^\infty r_j \cdot e^{ i \sum_{j=0}^\infty a_j}.
$$
Since $z$ is assumed to be real and $e_q(\bar \alpha) = \overline{e_q(\alpha)}$, we have
$$
 e_q(-iz) % =\prod_{j=0}^\infty r(-(1-q)zq^{j}) \cdot e^{ i \sum_{j=0}^\infty a(-(1-q)zq^{j})}
 =\prod_{j=0}^\infty r_j \cdot e^{- i  \sum_{j=0}^\infty a_j}.
$$
Note that $e_q(iz)$ and $e_q(-iz)$ are non-zero, by the assumption on $z$. Inserting the last two displays into the formulas for $\sin_q$ and $\cos_q$, we obtain
\begin{align*}
    \tan_q(z) &= \frac{\sin_q(z)}{\cos_q(z)}
 = \frac{ e_q(iz)-e_q(-iz)}{i(e_q(iz)+e_q(-iz))}
 =  \frac{ \exp( i \sum_{j=0}^\infty a_j)-\exp(- i \sum_{j=0}^\infty a_j)}{i(\exp(i\sum_{j=0}^\infty a_j)+\exp(-i\sum_{j=0}^\infty a_j)}
    \\
    & =  \frac{ \sin( \sum_{j=0}^\infty a_j)}{\cos(\sum_{j=0}^\infty a_j)}
 = \tan\left( \sum_{j=0}^\infty a_j\right)
=  \tan\left( \sum_{j=0}^\infty \arctan \bigl( (1-q)zq^{j})\bigr) \right). \qedhere
\end{align*}
\end{proof}
As a consequence of Lemma~\ref{lem:tan_q}, the function $\tan_q$ is one-to-one (and strictly increasing) from $[0,z_0(q))$ onto $[0,\infty)$. Hence one may define its inverse function
\begin{equation}
    \label{eq:def_arctan_q}
    \arctan_q:[0,\infty)\to [0,z_0(q)).
\end{equation}

The proof of Corollary~\ref{cor:asymptotics_theta<0} then follows directly, noting that the generating function of the persistence probabilities has its first positive singularity at a point where the denominator of \eqref{eq:main_formula_theta<0} vanishes. Since the function $\tan_q$ can take any value in $[0,\infty)$, the claim follows.

\subsection*{Acknowledgments}
KR is supported by the project RAWABRANCH (ANR-23-CE40-0008), funded by the French National Research Agency. Both authors would like to thank Thomas Simon for interesting discussions. The authors used ChatGPT (OpenAI) for brainstorming.

\bibliographystyle{plain} % We choose the "plain" reference style

\begin{thebibliography}{10}

\bibitem{AlBoRaSi-23}
Gerold Alsmeyer, Alin Bostan, Kilian Raschel, and Thomas Simon.
\newblock Persistence for a class of order-one autoregressive processes and
  {Mallows}-{Riordan} polynomials.
\newblock {\em Adv. Appl. Math.}, 150:52, 2023.
\newblock Id/No 102555.

\bibitem{aurzadabothe}
Frank Aurzada, Dieter Bothe, Pierre-\'Etienne Druet, Marvin Kettner, and
  Christophe Profeta.
\newblock Persistence exponents via perturbation theory: {G}aussian
  {MA}(1)-processes.
\newblock {\em Studia Math.}, 283(3):257--280, 2025.

\bibitem{kettner19}
Frank Aurzada and Marvin Kettner.
\newblock Persistence exponents via perturbation theory: {AR}(1)-processes.
\newblock {\em J. Stat. Phys.}, 177(4):651--665, 2019.

\bibitem{AuMuZe-21}
Frank Aurzada, Sumit Mukherjee, and Ofer Zeitouni.
\newblock Persistence exponents in {Markov} chains.
\newblock {\em Ann. Inst. Henri Poincar{\'e}, Probab. Stat.}, 57(3):1411--1441,
  2021.

\bibitem{AurzadaRaschel2025a}
Frank Aurzada and Kilian Raschel.
\newblock Persistence probabilities for {MA}(1) sequences with uniform
  innovations.
\newblock Preprint, arXiv:2507.04427, 2025.

\bibitem{Aurzada2015Persistence}
Frank Aurzada and Thomas Simon.
\newblock Persistence probabilities and exponents.
\newblock In Andreas Kyprianou, Ren{\'e} Schilling, and Thomas Simon, editors,
  {\em Lévy Matters V}, volume 2149 of {\em Lecture Notes in Mathematics},
  pages 183--224. Springer, Cham, 2015.

\bibitem{baumgarten14}
Christoph Baumgarten.
\newblock Survival probabilities of autoregressive processes.
\newblock {\em ESAIM Probab. Stat.}, 18:145--170, 2014.

\bibitem{BMMi-10}
Mireille Bousquet-M{\'e}lou and Marni Mishna.
\newblock Walks with small steps in the quarter plane.
\newblock In {\em Algorithmic probability and combinatorics. Papers from the
  AMS special sessions, Chicago, IL, USA, October 5--6, 2007 and Vancouver, BC,
  Canada, October 4--5, 2008}, pages 1--39. Providence, RI: American
  Mathematical Society (AMS), 2010.

\bibitem{Bray2013Persistence}
Alan~J. Bray, Satya~N. Majumdar, and Gr{\'e}gory Schehr.
\newblock Persistence and first-passage properties in nonequilibrium systems.
\newblock {\em Advances in Physics}, 62(3):225--361, 2013.

\bibitem{Krishna2016Persistence}
M.~Krishna and Manjunath Krishnapur.
\newblock Persistence probabilities in centered, stationary, {Gaussian}
  processes in discrete time.
\newblock {\em Indian J. Pure Appl. Math.}, 47(2):183--194, 2016.

\bibitem{larralde04}
Hern\'an Larralde.
\newblock A first passage time distribution for a discrete version of the
  {O}rnstein-{U}hlenbeck process.
\newblock {\em J. Phys. A}, 37(12):3759--3767, 2004.

\bibitem{MaDh-01}
Satya~N. Majumdar and Deepak Dhar.
\newblock Persistence in a stationary time series.
\newblock {\em Phys. Rev. E}, 64:046123, Sep 2001.

\bibitem{Metzler2014FirstPassage}
Ralf Metzler, Gleb Oshanin, and Sidney Redner, editors.
\newblock {\em First-Passage Phenomena and Their Applications}.
\newblock World Scientific Publishing Co. Pte. Ltd., Hackensack, NJ, 2014.

\bibitem{novikov08b}
Alexander Novikov and Nino Kordzakhia.
\newblock Martingales and first passage times of {$\rm AR(1)$} sequences.
\newblock {\em Stochastics}, 80(2-3):197--210, 2008.

\bibitem{novikov08a}
Alexander~A. Novikov.
\newblock Some remarks on the distribution of the first passage times and the
  optimal stopping of {$\rm AR(1)$}-sequences.
\newblock {\em Teor. Veroyatn. Primen.}, 53(3):458--471, 2008.

\bibitem{SALCEDOSANZ20221}
Sancho Salcedo-Sanz, David Casillas-P\'erez, Javier Del-Ser-Lorente, Carlos
  Casanova-Mateo, Lucas Cuadra, Maria Piles, and Gustau Camps-Valls.
\newblock Persistence in complex systems.
\newblock {\em Physics Reports}, 957:1--73, 2022.

\bibitem{Stanley1976BinomialPosets}
Richard~P. Stanley.
\newblock Binomial posets, m{\"o}bius inversion, and permutation enumeration.
\newblock {\em Journal of Combinatorial Theory, Series A}, 20(3):336--356,
  1976.

\bibitem{St-10}
Richard~P. Stanley.
\newblock A survey of alternating permutations.
\newblock In {\em Combinatorics and graphs. Selected papers based on the
  presentations at the 20th anniversary conference of IPM on combinatorics,
  Tehran, Iran, May 15--21, 2009. Dedicated to Reza Khosrovshahi on the
  occasion of his 70th birthday}, pages 165--196. Providence, RI: American
  Mathematical Society (AMS), 2010.

\bibitem{VyWa-23}
Vladislav Vysotsky and Vitali Wachtel.
\newblock Persistence of {AR}($1$) sequences with {Rademacher} innovations and
  linear mod $1$ transforms.
\newblock Preprint, {arXiv}:2305.10038, 2023.

\end{thebibliography}

\end{document}